\documentclass[10pt, reqno]{amsart}
\usepackage{amsmath, amsthm, amscd, amsfonts, amssymb, graphicx, color}
\usepackage[bookmarksnumbered, colorlinks, plainpages]{hyperref}
\hypersetup{colorlinks=true,linkcolor=red, anchorcolor=green, citecolor=cyan, urlcolor=red, filecolor=magenta, pdftoolbar=true}
\usepackage{mathrsfs}

\usepackage[OT2,T1]{fontenc}
\DeclareSymbolFont{cyrletters}{OT2}{wncyr}{m}{n}
\DeclareMathSymbol{\Sha}{\mathalpha}{cyrletters}{"58}

\newtheorem{theorem}{Theorem}[section]
\newtheorem{lemma}[theorem]{Lemma}

\newtheorem{corollary}[theorem]{Corollary}
\theoremstyle{definition}
\newtheorem{definition}[theorem]{Definition}
\newtheorem{example}[theorem]{Example}

\theoremstyle{remark}
\newtheorem{remark}[theorem]{Remark}
\numberwithin{equation}{section}

\begin{document}
\setcounter{page}{1}

\title[quantum arithmetic]{Quantum arithmetic}

\author[Nikolaev]
{Igor V. Nikolaev$^1$}

\address{$^{1}$ Department of Mathematics and Computer Science, St.~John's University, 8000 Utopia Parkway,  
New York,  NY 11439, United States.}
\email{\textcolor[rgb]{0.00,0.00,0.84}{igor.v.nikolaev@gmail.com}}





\begin{abstract}
We formalize  the quantum arithmetic,  i.e. a relationship  between number theory and operator algebras. 
 Namely, it is proved that  rational   projective varieties  are dual  to  the $C^*$-algebras with  real multiplication. 
 Our construction  fits  all axioms of the quantum arithmetic   conjectured by Manin and others. 
Applications  to   elliptic curves, Shafarevich-Tate groups of abelian varieties and  height
functions are  reviewed. 
\end{abstract}

\maketitle

\section{Introduction}
An interplay between algebraic geometry and number theory is known 
as  the arithmetic geometry  [Lang 1962] \cite{L}. 
It was discovered by  Serre, Sklyanin and others,  that 
multiplication  in algebraic geometry can be replaced by an 
operation which is no longer commutative [Stafford \& van ~den ~Bergh 2001] \cite{StaVdb1}.   
Independently, rings with the non-commutative multiplication were studied in the context of quantum 
mechanics [Murray \& von~Neumann 1936] \cite{MurNeu1}.  A part 
of the arithmetic geometry dealing with the non-commutative rings,  we call a quantum arithmetic \cite[Chapter 8]{N}. 
Such an arithmetic is  surprisingly efficient, e.g.  helping to reduce  deep
problems of number theory to the known facts of linear algebra, see Section 4.2 or \cite{Nik2}.

Yu. ~I.~Manin introduced axioms of the quantum arithmetic in the framework 
of his  real multiplication  program  [Manin 2004] \cite{Man1}.
The goal of such a program is  a solution to  Hilbert's 12th problem
for the real quadratic  fields,   \textit{ibid}.   The idea is to replace the  elliptic curve $\mathscr{E}_{CM}$
with complex multiplication (CM) by a quantum analog 
known as a noncommutative torus $\mathscr{A}_{\theta}$,  i.e. a $C^*$-algebra on the generators $u$ and $v$ satisfying the
commutation relation $vu=e^{2\pi i\theta}uv$ with $\theta\in \mathbf{R}$ being a constant  [Rieffel 1990] \cite{Rie1}. 
The  $\mathscr{A}_{\theta}$ is said to have real multiplication (RM), if  $\theta$ is a quadratic irrationality;  the corresponding notation   $\mathscr{A}_{RM}$. 
The axioms are as follows: (i) the map $\mathscr{E}_{CM} \to \mathscr{A}_{RM}$ is a functor 
from the category of elliptic curves  to such of  the noncommutative tori;  (ii) the Grothendieck semigroup $K_0^+(\mathscr{A}_{RM})$ of the algebra
$\mathscr{A}_{RM}$ (or, equivalently, the so-called dimension group $(K_0(\mathscr{A}_{RM}), K_0^+(\mathscr{A}_{RM}))$) 
gives rise to a ``pseudo-lattice''  $\Lambda_{RM}\subset\mathbf{R}$  replacing the usual lattice 
$L_{CM}\subset\mathbf{C}$ for which $\mathscr{E}_{CM}\cong \mathbf{C}/L_{CM}$.

The aim of our note is an extension of Manin's axioms to projective varieties $V(k)$ over  the number fields $k$. 
Namely, we replace  the $\mathscr{A}_{RM}$
by a more general algebra  $\mathscr{A}_V$ with RM, known as the Serre $C^*$-algebra of  $V(k)$
(Definition \ref{dfn1.1}).  We prove all  Manin's  axioms for $V(k)$ and  $\mathscr{A}_V$ (Theorem \ref{thm1.1}). 
We review applications of Theorem \ref{thm1.1} to the finiteness  conjecture for the Shafarevich-Tate groups  
(Theorem \ref{thm4.2}) and
height functions for the rational points of projective varieties (Theorem \ref{thm4.3}). 
To formalize our results,  let us recall  the following facts.

Let $V$ be an $n$-dimensional projective variety over the field of complex numbers $\mathbf{C}$.
 Recall \cite[Section 5.3.1]{N} that the Serre $C^*$-algebra  $\mathscr{A}_V$
 is  the norm closure of a self-adjoint representation of the twisted 
 homogeneous coordinate ring of  $V$  by the bounded linear operators acting on a Hilbert space;
 we refer the reader to  [Stafford \& van ~den ~Bergh 2001] \cite{StaVdb1} or Section 2.1 for the details.  
 Let $(K_0(\mathscr{A}_V), K_0^+(\mathscr{A}_V))$ be a dimension group 
 of the $C^*$-algebra $\mathscr{A}_V$ [Blackadar 1986] \cite[Section 6.1]{B}. 
  The triple  $(\Lambda, [I], K)$ stays for a dimension group generated by
   the ideal class $[I]$ of an order $\Lambda\subseteq O_K$ 
  in the ring of integers of a number field $K$ [Effros 1981] \cite[Chapter 6]{E},
  [Handelman 1981] \cite{Han1} or \cite[Theorem 3.5.4]{N}.  
  \begin{definition}\label{dfn1.1}
  The Serre $C^*$-algebra $\mathscr{A}_V$ is said to have real multiplication  by 
  the triple
   $(\Lambda, [I], K)$,   if  there exists  an isomorphism of the dimension group
   $(K_0(\mathscr{A}_V), K_0^+(\mathscr{A}_V))\cong (\Lambda, [I], K)$,
   where   $\Lambda\subseteq O_K$ is an order,   $[I]\subset\Lambda$ is  an ideal class  and 
   $K$ is  a  number field. 
  \end{definition}
\begin{example}\label{exm1.2}
{\bf (\cite{Nik1})}
Let $V\cong\mathscr{E}_{CM}$ be an elliptic curve with complex multiplication by the triple $(L, [I], k)$,
where $L=\mathbf{Z}+fO_k$ is an order of conductor $f\ge 1$ in the ring $O_k$ of an
imaginary quadraitic field $k\cong \mathbf{Q}(\sqrt{-d})$ and $[I]$ an ideal class in $L$.   
Then  $\mathscr{A}_V\cong \mathscr{A}_{RM}$ is a noncommutative torus with real 
multiplication by the triple $(\Lambda, [I], K)$,
where $[I]$ an ideal class in the order $L=\mathbf{Z}+f'O_k$ of conductor $f'\ge 1$ in the ring  $O_k$ of an
imaginary quadraitic field $K\cong\mathbf{Q}(\sqrt{d})$, such that 
$f'$ is the least integer satisfying the equation 
$|Cl(\mathbf{Z}+f'O_{K})|= |Cl(\mathbf{Z}+fO_{k})|$,
where  $Cl (R)$ is the class group of the ring $R$. 
\end{example}

Let $\mathscr{M}_V$ be the moduli space of variety $V$ and $m=\dim_{\mathbf{C}} \mathscr{M}_V$. 
 Denote by $O_K$ the ring of integers of a number field  $K$ of degree 
  $\deg ~(K|\mathbf{Q})=2m$. 
Our main result can be formulated as follows.
\begin{theorem}\label{thm1.1}
 The Serre $C^*$-algebra $\mathscr{A}_V$ has real multiplication 
by a triple $(\Lambda, [I], K)$, if and only if, 
the projective variety $V$ is defined over a number field $k$.
\end{theorem}
\begin{remark}\label{rmk1.4}
Theorem \ref{thm1.1} is an existence result  proved by contradiction (Section 3).
Such a proof does not give  explicit formulas for the number field $K$
in terms of $k$,  except for the special case of complex multiplication (\cite{Nik1}
or Section 4.1).  
\end{remark}

\bigskip
Our note is organized as follows.  A brief review of the preliminary facts is 
given in Section 2. Theorem \ref{thm1.1} 
is proved in Section 3.  Some applications of Theorem \ref{thm1.1} are reviewed
in Section 4.

\section{Preliminaries}
We briefly review the Serre $C^*$-algebras, dimension groups and twists.   
We refer the reader to  [Effros 1981] \cite{E},   [Handelman 1981] \cite{Han1}, 
  [Stafford \& van ~den ~Bergh 2001] \cite{StaVdb1} 
  and \cite[Section 5.3.1]{N} for a detailed exposition.

\subsection{Serre $C^*$-algebras}
Let $V$ be a projective variety over the field $k$.  Denote by $\mathcal{L}$ an invertible
sheaf of the linear forms on $V$.  If $\sigma$ is an automorphism of $V$,  then
the pullback of $\mathcal{L}$ along $\sigma$ will be denoted by $\mathcal{L}^{\sigma}$,
i.e. $\mathcal{L}^{\sigma}(U):= \mathcal{L}(\sigma U)$ for every $U\subset V$. 
The graded $k$-algebra
$B(V, \mathcal{L}, \sigma)=\bigoplus_{i\ge 0} H^0\left(V, ~\mathcal{L}\otimes \mathcal{L}^{\sigma}\otimes\dots
\otimes  \mathcal{L}^{\sigma^{ i-1}}\right)$
is called a  twisted homogeneous coordinate ring of $V$ [Stafford \& van den Bergh 2001]  \cite{StaVdb1}.  Such a ring is 
always non-commutative,  unless the automorphism $\sigma$ is trivial. 
A multiplication of sections of  $B(V, \mathcal{L}, \sigma)=\oplus_{i=1}^{\infty} B_i$ is defined by the 
rule  $ab=a\otimes b$,   where $a\in B_m$ and $b\in B_n$.
An invertible sheaf $\mathcal{L}$ on $V$  is called $\sigma$-ample, if for 
every coherent sheaf $\mathcal{F}$ on $V$,
 the cohomology group $H^k(V, ~\mathcal{L}\otimes \mathcal{L}^{\sigma}\otimes\dots
\otimes  \mathcal{L}^{\sigma^{ n-1}}\otimes \mathcal{F})$  vanishes for $k>0$ and
$n>>0$.   If $\mathcal{L}$ is a $\sigma$-ample invertible sheaf on $V$,  then
$Mod~(B(V, \mathcal{L}, \sigma)) / ~Tors ~\cong ~Coh~(V)$,
where  $Mod$ is the category of graded left modules over the ring $B(V, \mathcal{L}, \sigma)$,
$Tors$ is the full subcategory of $Mod$ of the torsion  modules and  $Coh$ is the category of 
quasi-coherent sheaves on a scheme $V$.  In other words, the $B(V, \mathcal{L}, \sigma)$  is  
a coordinate ring of the variety $V$.

Let $R$ be a commutative  graded ring,  such that $V=Proj~(R)$.  
Denote by $R[t,t^{-1}; \sigma]$
the ring of skew Laurent polynomials defined by the commutation relation
$b^{\sigma}t=tb$  for all $b\in R$, where $b^{\sigma}$ is the image of  $b$ under automorphism 
$\sigma$.  It is known, that $R[t,t^{-1}; \sigma]\cong B(V, \mathcal{L}, \sigma)$ [Stafford \& van den Bergh 2001]  \cite[Section 5]{StaVdb1}.
Let $\mathcal{H}$ be a Hilbert space and   $\mathscr{B}(\mathcal{H})$ the algebra of 
all  bounded linear  operators on  $\mathcal{H}$.
For a  ring of skew Laurent polynomials $R[t, t^{-1};  \sigma]$,  
 consider a homomorphism
$\rho: R[t, t^{-1};  \sigma]\longrightarrow \mathscr{B}(\mathcal{H})$. 
Recall  that  $\mathscr{B}(\mathcal{H})$ is endowed  with a $\ast$-involution;
the involution comes from the scalar product on the Hilbert space $\mathcal{H}$. 
The representation $\rho$ is called  $\ast$-coherent,   if
(i)  $\rho(t)$ and $\rho(t^{-1})$ are unitary operators,  such that
$\rho^*(t)=\rho(t^{-1})$ and 
(ii) for all $b\in R$ it holds $(\rho^*(b))^{\sigma(\rho)}=\rho^*(b^{\sigma})$, 
where $\sigma(\rho)$ is an automorphism of  $\rho(R)$  induced by $\sigma$. 
Whenever  $B=R[t, t^{-1};  \sigma]$  admits a $\ast$-coherent representation,
$\rho(B)$ is a $\ast$-algebra.  The norm closure of  $\rho(B)$  is   a   $C^*$-algebra
   denoted  by $\mathscr{A}_V$.  We  refer to  $\mathscr{A}_V$  as   the    Serre $C^*$-algebra
 of  $V$ \cite[Section 5.3.1]{N}.

\subsection{Dimension groups}
The  $C^*$-algebra $\mathcal{A}$ is an algebra over $\mathbf{C}$ with a norm
$a\mapsto ||a||$ and an involution $a\mapsto a^*$ such that
it is complete with respect to the norm and $||ab||\le ||a||~ ||b||$
and $||a^*a||=||a||^2$ for all $a,b\in \mathcal{A}$.
Any commutative $C^*$-algebra is  isomorphic
to the algebra $C_0(X)$ of continuous complex-valued
functions on some locally compact Hausdorff space $X$. 
Any other  algebra $\mathcal{A}$ can be thought of as  a non-commutative  
topological space.   By $M_{\infty}(\mathcal{A})$ 
one understands the algebraic direct limit of the $C^*$-algebras 
$M_n(\mathcal{A})$ under the embeddings $a\mapsto ~\mathbf{diag} (a,0)$. 
The direct limit $M_{\infty}(\mathcal{A})$  can be thought of as the $C^*$-algebra 
of infinite-dimensional matrices whose entries are all zero except for a finite number of the
non-zero entries taken from the $C^*$-algebra $\mathcal{A}$.
Two projections $p,q\in M_{\infty}(\mathcal{A})$ are equivalent, if there exists 
an element $v\in M_{\infty}(\mathcal{A})$,  such that $p=v^*v$ and $q=vv^*$. 
The equivalence class of projection $p$ is denoted by $[p]$.   
We write $V(\mathcal{A})$ to denote all equivalence classes of 
projections in the $C^*$-algebra $M_{\infty}(\mathcal{A})$, i.e.
$V(\mathcal{A}):=\{[p] ~:~ p=p^*=p^2\in M_{\infty}(\mathcal{A})\}$. 
The set $V(\mathcal{A})$ has the natural structure of an abelian 
semi-group with the addition operation defined by the formula 
$[p]+[q]:=\mathbf{diag}(p,q)=[p'\oplus q']$, where $p'\sim p, ~q'\sim q$ 
and $p'\perp q'$.  The identity of the semi-group $V(\mathcal{A})$ 
is given by $[0]$, where $0$ is the zero projection. 
By the $K_0$-group $K_0(\mathcal{A})$ of the unital $C^*$-algebra $\mathcal{A}$
one understands the Grothendieck group of the abelian semi-group
$V(\mathcal{A})$, i.e. a completion of $V(\mathcal{A})$ by the formal elements
$[p]-[q]$.  The image of $V(\mathcal{A})$ in  $K_0(\mathcal{A})$ 
is a positive cone $K_0^+(\mathcal{A})$ defining  the order structure $\le$  on the  
abelian group  $K_0(\mathcal{A})$. The pair   $\left(K_0(\mathcal{A}),  K_0^+(\mathcal{A})\right)$
is known as a dimension group of the $C^*$-algebra $\mathcal{A}$. 
The scale $\Sigma(\mathcal{A})$ is the image in $K_0^+(\mathcal{A})$
of the equivalence classes of projections in the $C^*$-algebra $\mathcal{A}$. 
The $\Sigma(\mathcal{A})$ is a generating, hereditary and directed subset 
of  $K_0^+(\mathcal{A})$, i.e. (i) for each $a\in K_0^+(\mathcal{A})$ 
there exist $a_1,\dots, a_r\in\Sigma(\mathcal{A})$ such that 
$a=a_1+\dots+a_r$; (ii) if $0\le a\le b\in \Sigma(\mathcal{A})$, then $a\in\Sigma(\mathcal{A})$
and (iii) given $a,b\in\Sigma(\mathcal{A})$ there exists $c\in\Sigma(\mathcal{A})$,
such that $a,b\le c$.   Each  scale  can always be written as 
$\Sigma(\mathcal{A})=\{a\in K_0^+(\mathcal{A}) ~|~0\le a\le u\}$,
where $u$ is an  order unit of  $K_0^+(\mathcal{A})$.  
The pair  $\left(K_0(\mathcal{A}),  K_0^+(\mathcal{A})\right)$ and the
triple  $\left(K_0(\mathcal{A}),  K_0^+(\mathcal{A}), \Sigma(\mathcal{A})\right)$
are invariants of the Morita equivalence and isomorphism class of the 
$C^*$-algebra $\mathcal{A}$, respectively.

\subsection{Twists}
Let $V$ be a complex projective variety given by a  homogeneous coordinate ring $\mathscr{A}$. 
If $k\subset\mathbf{C}$ is a subfield of complex numbers and $V(k)$ is the set of  $k$-points of $V$, 
then the isomorphisms of $V(k)$ over $\mathbf{C}$ cannot be restricted to the field $k$
in general.   When such a restriction fails,  the variety $V'(k)$ is called a  twist of $V(k)$. 
Equivalently, the varieties $V(k) $ and  $V'(k)$ are not isomorphic over $k$,  yet  they are 
isomorphic over  $\mathbf{C}$.  The twists of $V(k)$ are  classified  in terms of  the Galois cohomology   
 [Serre 1997]  \cite[p. 123]{S}.
 
 Let $G$ be a group. The set $\mathbf{A}$ is called a $G$-set, if 
$G$ acts on $\mathbf{A}$ on the left continuously. 
If $\mathbf{A}$ is a group and $G$ acts on $\mathbf{A}$ by the group
morphisms, then $\mathbf{A}$ is called a $G$-group. In particular,
if $\mathbf{A}$ is abelian, one gets a  $\mathbf{G}$-module.

If $\mathbf{A}$ is a $G$-group,  then a 1-cocycle of $G$ in $\mathbf{A}$ 
is a map $s\mapsto a_s$ of $G$ to $A$ which is continuous and such that
$a_{st}=a_s a_t$ for all $s,t\in G$.  The set of all 1-cocycles is denoted by
$Z^1(G, \mathbf{A})$.  Two cocycles $a$ and $a'$ are said to be  cohomologous, 
 if there exists $b\in\mathbf{A}$ such that $a_s'=b^{-1}a_s b$. The quotient  of 
 $Z^1(G, \mathbf{A})$ by this equivalence relation is called the first 
  cohomology set  and is denoted by $H^1(G,\mathbf{A})$. 
The class of the unit cocycle is a distinguished element $\mathbf{1}$ in
the $H^1(G,\mathbf{A})$.  Notice that in general  there is no composition law on the 
set $H^1(G,\mathbf{A})$. If $\mathbf{A}$ is an abelian group, the set 
$H^1(G,\mathbf{A})$ is a cohomology group.

If $G$ is a profinite group, then 
\begin{equation}\label{eq2.2}
H^1(G,\mathbf{A})=\varinjlim H^1(G/U, \mathbf{A}^U),
\end{equation}
where $U$ runs through the set of open normal subgroups of $G$ and 
$\mathbf{A}^U$ is a subset of $\mathbf{A}$ fixed under action of $U$. 
The maps $H^1(G/U, \mathbf{A}^U)\to H^1(G,\mathbf{A})$ are injective. 

Let $k$ be a number field and $\bar k$ the algebraic closure of $k$. Denote by 
$Gal~(\bar k|k)$ the profinite Galois group of $\bar k$.  Let $V(k)$ be a projective 
variety over $k$ and $Aut ~V(k)$  the group of the $\bar k$-automorphisms
of $V(k)$. 
\begin{theorem}\label{thm2.1}
{\bf  [Serre 1997]  \cite[p. 124]{S}}
There exits a bijective correspondence between the twists of $V(k)$ and the set
$H^1(Gal~(\bar k|k), ~Aut ~V(k))$. 
\end{theorem}

\section{Proof of theorem \ref{thm1.1}}
{\bf Part I.} Let us prove ``if'' part of  Theorem \ref{thm1.1}.
For the sake of clarity, we outline the main ideas.
Our proof is by contradiction. Assume to the contrary
that  $V$ is defined over a number field $k$,  but the corresponding Serre $C^*$-algebra
$\mathscr{A}_V$  has no real multiplication. 
Recall that each Serre $C^*$-algebra is a crossed product $\mathscr{A}_V\cong C(V)\rtimes_{\alpha}\mathbf{Z}$ by an automorphism $\alpha$ of $V$ 
\cite[Lemma 5.3.2]{N}.
Consider the Elliott-Serre $C^*$-algebra  $\mathbb{A}_V$ (Definition \ref{dfn3.2}),  i.e. an enveloping $AF$-algebra 
 of the crossed product  $C(V)\rtimes_{\alpha} \mathbf{Z}$
[Pimsner 1983] \cite[Theorem 9 (3)]{Pim1}. 
One gets an inclusion of the dimension groups   $(K_0(\mathscr{A}_V), K_0^+(\mathscr{A}_V))\subseteq  (K_0(\mathbb{A}_V), K_0^+(\mathbb{A}_V))$. 
It is known that  $\mathbb{A}_V$  (and therefore $\mathscr{A}_V$) has RM or no RM depending on the Bratteli diagram of 
$\mathbb{A}_V$ being periodic or aperiodic,  respectively  [Effros 1981] \cite[Chapter 6]{E}. 
Since we assumed that  $\mathscr{A}_V$ has no RM, one concludes that variety $V(k)$ has infinitely many pairwise 
non-isomorphic twists \cite[Corollary 1.2]{Nik3}. (This fact follows from aperiodicity of  the diagram, where twists are in a one-to-one 
correspondence with the blocks of diagram.) One gets a  contradiction with  Theorem \ref{thm2.1} implying that the number 
of twists of $V(k)$ is always finite. Therefore the Serre $C^*$-algebra    $\mathscr{A}_V$ has RM by a triple $(\Lambda, [I], K)$.
We pass to a detailed argument by splitting the proof in a series of lemmas.

\begin{lemma}\label{lm3.1}
For every Serre $C^*$-algebra  $\mathscr{A}_V$ there exists a unique
enveloping  $AF$-algebra $\mathbb{A}_V$, i.e.

\medskip
(i) $\mathscr{A}_V\subset \mathbb{A}_V$ is  a unital inclusion of the $C^*$-algebras;

\smallskip
(ii) $(K_0(\mathscr{A}_V), K_0^+(\mathscr{A}_V))\subseteq  (K_0(\mathbb{A}_V), K_0^+(\mathbb{A}_V))$ 
is an inclusion of the  dimension groups.
\end{lemma} 
\begin{proof}
(i)  Recall that every Serre $C^*$-algebra   $\mathscr{A}_V$ is isomorphic to 
the crossed pro\-duct $C^*$-algebra $C(V)\rtimes_{\alpha_T}\mathbf{Z}$,
where $C(V)$ is the $C^*$-algebra of all continuous complex-valued functions 
on $V$ and $\alpha_T$ is a $*$-coherent automorphism of $C(V)$ 
\cite[Lemma 5.3.2]{N}.  Moreover, it is easy to verify that the corresponding homeomorphism
$T: V\to V$ is pseudo-non-wandering for each point of $V$,
i.e. $V=V(T)$ [Pimsner 1983] \cite[Definition 2]{Pim1}.
One can apply an equivalence of conditions (1) and (3) in [Pimsner 1983] \cite[Theorem 9]{Pim1}
to obtain a unital embedding of the crossed product  $C(V)\rtimes_{\alpha_T}\mathbf{Z}$
into an $AF$-algebra  $\mathbb{A}_V$.  Moreover, the algebra $\mathbb{A}_V$ is uniquely defined by 
the crossed product  $C(V)\rtimes_{\alpha_T}\mathbf{Z}$ in view of  Pimsner's construction,  {\it ibid.}
The classification of the $AF$-algebras is due to 
[Elliott 1976] \cite{Ell1};  hence the following notation.
\begin{definition}\label{dfn3.2}
The $AF$-algebra $\mathbb{A}_V$ will be called an Elliott-Serre $C^*$-algebra.   
\end{definition}

\bigskip
(ii) Let us prove an inclusion of the dimension groups  $(K_0(\mathscr{A}_V), K_0^+(\mathscr{A}_V))\subseteq  (K_0(\mathbb{A}_V), K_0^+(\mathbb{A}_V))$.  
Recall that $h: \mathscr{A}_V\to\mathbb{A}_V$ is an injective homomorphism of the $C^*$-algebras. 
On the other hand, the dimension group is a functor, i.e. any injective homomorphism  $h: \mathscr{A}_V\to\mathbb{A}_V$
induces a pair of injective homomorphisms $h_*:  K_0(\mathscr{A}_V)\to K_0(\mathbb{A}_V)$ 
and  $h_*^+:  K_0^+(\mathscr{A}_V)\to K_0^+(\mathbb{A}_V)$  [Blackadar 1986] \cite[Section 2.2]{B}. 
Since $h_*$ and $h_*^+$ can be  surjective in general, one gets a non-strict inclusion of the 
dimension groups  $(K_0(\mathscr{A}_V), K_0^+(\mathscr{A}_V))\subseteq  (K_0(\mathbb{A}_V), K_0^+(\mathbb{A}_V))$.

\bigskip
Lemma \ref{lm3.1} is proved.
\end{proof}

\begin{lemma}\label{lm3.2}
The Bratteli diagram of  $\mathbb{A}_V$ is eventually periodic, if and only if, 
the Serre $C^*$-algebra $\mathscr{A}_V$ has real multiplication. 
\end{lemma} 
\begin{proof}
(i) Let us first prove that the Bratteli diagram of  the Elliott-Serre $C^*$-algebra  $\mathbb{A}_V$ is eventually periodic, if and only if, 
the  $\mathbb{A}_V$ has real multiplication. 
Indeed, if the Bratteli diagram of the $AF$-algebra  $\mathbb{A}_V$ is eventually periodic, then the order-isomorphism
 of dimension group   $(K_0(\mathbb{A}_V), K_0^+(\mathbb{A}_V))$ is defined by the triple  $(\Lambda, [I], K)$ 
 [Handelman 1981] \cite[Section 1]{Han1}.  In other words, the Elliott-Serre $C^*$-algebra $\mathbb{A}_V$
 has real multiplication (Definition \ref{dfn1.1}).  Conversely, assume that  $\mathbb{A}_V$ has real multiplication,
 i.e. $(K_0(\mathbb{A}_V), K_0^+(\mathbb{A}_V))\cong (\Lambda, [I], K)$.  One can apply  [Handelman 1981] \cite[Theorem II, items (ii) and (v)]{Han1}
 saying that  purely periodic Bratteli diagram are in one-to-one correspondence with the triples $(\Lambda, [I], K)$, where $\Lambda\cong End_c G$ and 
 $G\cong (K_0(\mathbb{A}_V), K_0^+(\mathbb{A}_V))$.  The eventually periodic Bratteli diagrams are order-isomorphic to the latter, {\it ibid.} 
 Item (i) is proved.

\bigskip
(ii) It remains to show that the real multiplication of the Elliott-Serre $C^*$-algebra  $\mathbb{A}_V$ implies such 
of the Serre $C^*$-algebra $\mathscr{A}_V$.  Recall that  by item (ii) of Lemma \ref{lm3.1},  
one gets an inclusion of the dimension groups  $(K_0(\mathscr{A}_V), K_0^+(\mathscr{A}_V))\subseteq  (K_0(\mathbb{A}_V), K_0^+(\mathbb{A}_V))$.  
It is clear that if the dimension group $(K_0(\mathbb{A}_V), K_0^+(\mathbb{A}_V))$ has real multiplication, so does any subgroup of it. 
Indeed, if $(K_0(\mathbb{A}_V), K_0^+(\mathbb{A}_V))\cong (\Lambda, [I], K)$,  then $(K_0(\mathscr{A}_V), K_0^+(\mathscr{A}_V))\cong  (\Lambda', [I], K)$,
where $\Lambda'\subseteq\Lambda\subset\mathbf{R}$ is a  finite index subgroup  of the additive group of $\Lambda$.    
In other words, the dimension group  $(K_0(\mathscr{A}_V), K_0^+(\mathscr{A}_V))$ has real multiplication by $(\Lambda', [I], K).$

\bigskip
Lemma \ref{lm3.2} is proved.
\end{proof}

\begin{lemma}\label{lm3.3}
There exists a one-to-one correspondence between twists of the variety $V(k)$ and irreducible blocks of the 
Bratteli diagram of the Elliott-Serre $C^*$-algebra  $\mathbb{A}_V$. 
\end{lemma} 
\begin{proof}
(i) Let $V(k)$ and $V'(k)$ be complex projective varieties over 
a number field $k\subset\mathbf{C}$.  
Recall \cite[Corollary 1.2]{Nik3} that: 
(i) $V(k)$ and $V'(k)$ are isomorphic over $k$  if and only if the Serre $C^*$-algebras
$\mathscr{A}_V\cong\mathscr{A}_{V'}$ are isomorphic
and  (ii)  $V(k)$ and $V'(k)$ are isomorphic over $\mathbf{C}$  if and only if the Serre  $C^*$-algebras
$\mathscr{A}_V$ and $\mathscr{A}_{V'}$ are Morita equivalent, i.e. $\mathscr{A}_V\otimes\mathcal{K}\cong\mathscr{A}_{V'}\otimes\mathcal{K}$,
where $\mathcal{K}$ is the $C^*$-algebra of compact operators. 
Using the inclusion  $\mathscr{A}_V\subset \mathbb{A}_V$ (Lemma \ref{lm3.1}),
one can extend the Morita equivalence between the Serre  $C^*$-algebras
$\mathscr{A}_V$ and $\mathscr{A}_{V'}$ to such between the $AF$-algebras $\mathbb{A}_V$ and $\mathbb{A}_{V'}$.

\bigskip
(ii) Recall that the $AF$-algebras $\mathbb{A}_V$ and $\mathbb{A}_{V'}$ are Morita equivalent, if and only if, 
their Bratteli diagrams coincide except for a finite number of the irreducible blocks. 
In particular, taking away one block at a time from the left of the infinite Bratteli diagram
\begin{equation}\label{eq3.1}
\mathbf{Z}^n\buildrel\rm B_1\over\longrightarrow \mathbf{Z}^n
   \buildrel\rm B_2\over\longrightarrow
    \mathbf{Z}^n
   \buildrel\rm B_3\over\longrightarrow
   \dots,
\end{equation}
gives us an infinite sequence of the $AF$-algebras which are  Morita equivalent to  $\mathbb{A}_V$.
Clearly, these $AF$-algebras can be indexed by the blocks $B_i$. 

\bigskip
(iii)  It remains to apply item (i)  to the obtained Morita equivalence classes of the $AF$-algebra   $\mathbb{A}_V$
by noticing that they correspond to the twists of the variety $V(k)$.

\bigskip
Lemma \ref{lm3.3} is proved.
\end{proof}

\begin{corollary}\label{cor3.4}
The Bratteli diagram of  $\mathbb{A}_V$ is eventually periodic 
whenever variety $V$ is defined over a number field $k\subset\mathbf{C}$. 
\end{corollary} 
\begin{proof}
(i) To the contrary, assume that the variety $V(k)$ is defined over a number field $k$,  
but the Bratteli diagram of the Elliott-Serre $C^*$-algebra   $\mathbb{A}_V$ is 
aperiodic.  
In this case one gets an infinite number of the pairwise distinct irreducible
blocks $B_i$ in the sequence (\ref{eq3.1}). By Lemma \ref{lm3.3}, one obtains
infinitely many pairwise non-isomorphic twists of the variety $V(k)$.    
The latter contradicts  Theorem \ref{thm2.1} saying that the set 
$H^1(Gal~(\bar k|k), ~Aut ~V(k))$ is finite for each finite Galois extension $k'\subset\bar k$
of the number field $k$. 

\bigskip
(ii) We conclude therefore that the sequence $B_i$ must stabilize, i.e. there exists an integer $N>0$, 
 such that  $B_i=Const$ for all $i\ge N$.  In other words, the Bratteli diagram    of  the 
 Elliott-Serre $C^*$-algebra $\mathbb{A}_V$ is eventually periodic. 
 
 \bigskip
 Corollary \ref{cor3.4} is proved.
  \end{proof}

\bigskip
Part I of Theorem \ref{thm1.1} follows from Lemma \ref{lm3.2} and Corollary \ref{cor3.4}.

\bigskip
{\bf Part II.} Let us prove ``only if''  part of  Theorem \ref{thm1.1}.
To the contrary, assume that the Serre $C^*$-algebra  $\mathscr{A}_V$ has real multiplication,
but the projective variety $V$ is defines over the field of complex numbers $\mathbf{C}$. 
Consider the enveloping Elliott-Serre $C^*$-algebra  $\mathbb{A}_V$. 
Since $V$ is defined over $\mathbf{C}$, there are infinitely many projective varieties $V'$
isomorphic to $V$ over $\mathbf{C}$. In view of  \cite[Corollary 1.2]{Nik3},
there are infinitely many Serre $C^*$ algebras $\mathscr{A}_{V'}$ which are 
Morita equivalent to $\mathscr{A}_V$.  By Lemma \ref{lm3.3}, the Bratteli diagram
of the $AF$-algebra $\mathbb{A}_V$ must be aperiodic in this case. Hence 
the Serre $C^*$-algebra has no real multiplication (Lemma \ref{lm3.2}). 
The contradiction finishes the proof of Part II of Theorem \ref{thm1.1}.

\bigskip
Theorem \ref{thm1.1} is proved.

\section{Applications}
We review some applications of Theorem \ref{thm1.1}
to  elliptic curves over number fields, Shafarevich-Tate groups 
of abelian varieties and height functions.

\subsection{Elliptic curves over number fields}
The simplest non-trivial case  when $V\cong \mathscr{E}(k)$ is an elliptic curve
over the number field $k$.  Recall that non-singular elliptic curves are identified via the Weierstrass $\wp$-function with
the complex tori $\mathbf{C}/L_{\tau}$, where $L_{\tau}=\mathbf{Z}+\mathbf{Z}\tau$
is a lattice defined by modulus $\tau\in\mathbf{C}$ such that $Im~\tau>0$. 
In particular, the moduli space $\mathscr{M}_V$ of  elliptic curves  is one-dimensional, i.e.
\begin{equation}
m=\dim_{\mathbf{C}}\mathscr{M}_V=1.
\end{equation}
We calculate  $\deg ~(K|\mathbf{Q})=2m=2$, i.e the number field $K$ is a real quadratic
extension of $\mathbf{Q}$.   We apply Theorem \ref{thm1.1} to conclude that 
there exists an order $\Lambda\subseteq O_K$,  such that for an ideal class $[I]\subset\Lambda$
the following dimension groups are isomorphic:
\begin{equation}\label{eq4.2}
(K_0(\mathscr{A}_{\mathscr{E}(k)}), K_0^+(\mathscr{A}_{\mathscr{E}(k)}))\cong (\Lambda, [I], K).
\end{equation}

\bigskip
In fact, the Serre $C^*$-algebra $\mathscr{A}_{\mathscr{E}(k)}$ and its enveloping AF-algebra
 $\mathbb{A}_{\mathscr{E}(k)}$ admit a remarkable independent description. 
Indeed, all dimension groups of rank $2$ arise as the $K_0$-groups of the Effros-Shen algebras $\mathbb{A}_{\theta}$
[Effros \& Shen 1980] \cite{EfrShe1}. Such an algebra is given by the continued fraction of the real constant $\theta$.
In particular, the condition (\ref{eq4.2}) implies $\theta\in K$ is a quadratic irrational number, i.e. the continued fraction 
of $\theta$ is eventually periodic.  Thus  the Elliott-Serre $C^*$-algebra  $\mathbb{A}_{\mathscr{E}(k)}\cong \mathbb{A}_{RM}$,
where $\mathbb{A}_{RM}$ is the Effros-Shen algebra given by an eventually periodic continued fraction. 
Moreover, each  $\mathbb{A}_{\theta}$ contains a copy of the noncommutative torus $\mathscr{A}_{\theta}$ having the same
dimension group [Rieffel 1990] \cite{Rie1}. We conclude therefore that $\mathscr{A}_{\mathscr{E}(k)}\cong \mathscr{A}_{RM}$,
where $\mathscr{A}_{RM}$ is a noncommutative torus with real multiplication.

\bigskip
Following example \ref{exm1.2}, let us consider a special case of elliptic curves with complex multiplication
$\mathscr{E}_{CM}$.  This case correspond to the lattices $L_{\tau}$,  where  $\tau\in\mathbf{Q}(\sqrt{-d})$
an imaginary quadratic number.  It is known that such elliptic curves are defined over the field 
$k\cong\mathbf{Q}(j( \mathscr{E}_{CM}))$,
where $j( \mathscr{E}_{CM})$ is the $j$-invariant of  $\mathscr{E}_{CM}$. 
The triple  $(\Lambda, [I], K)$ in this case can be given explicitly.
\begin{theorem}{\bf (\cite{Nik1})}
Let $\mathscr{E}_{CM}$ be an elliptic curve with complex multiplication 
by an order of conductor $f$ in the field $\mathbf{Q}(\sqrt{-d})$. 
Then $\Lambda$ is an order of conductor $f'$ in the field  $\mathbf{Q}(\sqrt{d})$,
where $f'$ is the least integer satisfying the equation 
$|Cl(\mathbf{Z}+f'O_{\mathbf{Q}(\sqrt{d})})|= |Cl(\mathbf{Z}+fO_{\mathbf{Q}(\sqrt{-d})})|$
and $Cl (R)$ is the class group of the ring $R$. 
\end{theorem}

\subsection{Shafarevich-Tate groups of abelian varieties}
Recall that if a diophantine  equation has an integer solution, then using the reduction
modulo  prime $p$, one gets a solution of the 
 equation lying in the finite field $\mathbf{F}_p$ and  a solution 
in the field of real numbers  $\mathbf{R}$. The equation is said to satisfy
the Hasse principle, if the converse  is true.  
For instance,  the quadratic equations  satisfy the Hasse principle,
while  the equation $x^4-17=2y^2$ has a solution over $\mathbf{R}$ 
and  each  $\mathbf{F}_p$,   but no rational solutions.

Let $A(k)$ be an abelian variety over the number field $k$
which we assume to be simple, i.e. the $A(k)$ has no proper sub-abelian varieties 
over $k$. The Shafarevich-Tate group $\Sha (A)$ measures the failure of the 
Hasse principle for $A(k)$. 
Denote by $k_v$ the completion of $k$ at the (finite or infinite) place $v$. 
Consider the Weil-Ch\^atelet group $WC(A(k))$ of the abelian 
variety $A(k)$ and the group homomorphism:
\begin{equation}\label{eq4.3}
WC(A(k))\rightarrow \prod_v WC(A(k_v)). 
\end{equation}
The Shafarevich-Tate group $\Sha (A)$ 
 is defined as the kernel of homomorphism (\ref{eq4.3}). 
The  variety $A(k)$ satisfies the Hasse principle, 
if and only if,  the group $\Sha (A)$ is trivial.

To calculate  $\Sha (A)$, we apply Theorem \ref{thm1.1} with  $V\cong A(k)$
and  get a triple $(\Lambda, [I], K)$.   Recall that the moduli space $\mathscr{M}_V$ of  
an $n$-dimensional abelian variety  has complex dimension $\frac{1}{2} n(n+1)$, i.e.
\begin{equation}
m=\dim_{\mathbf{C}}\mathscr{M}_V=\frac{n(n+1)}{2}.
\end{equation}
We calculate  $\deg ~(K|\mathbf{Q})=2m=n(n+1)$.
Let us recast   (\ref{eq4.3}) in terms of  the triple $(\Lambda, [I], K)$. 
Recall that an isomorphism class of the Serre $C^*$-algebra $\mathscr{A}_V$ is defined 
by the scaled dimension group $\left(K_0(\mathscr{A}_{V}, K_0^+(\mathscr{A}_{V}),
\Sigma(\mathscr{A}_{V}\right)$ consisting of the $K_0$-group, the positive cone $K_0^+$ and 
the scale $\Sigma$ of   the algebra $\mathscr{A}_{V}$ [Blackadar 1986] \cite[Section 6]{B}.
  Using  the Minkowski question-mark function, it can be shown that  $\Sigma(\mathscr{A}_{V})$  
 is  a torsion group, such that   $WC(\mathscr{A}_K)\cong \Sigma(\mathscr{A}_{V})$. 
The RHS  of (\ref{eq4.3}) corresponds to the crossed product $C^*$-algebra
$\mathscr{A}_{V}\rtimes_{L_v}\mathbf{Z}$,  where $L_v$ is the Frobenius endomorphsm of $\mathscr{A}_{V}$;
we refer the reader to \cite{Nik2} for the notation and details.   
It is proved,  that $\prod_v WC(\mathscr{A}_{K_v})\cong \prod_v K_0\left(\mathscr{A}_{V}\rtimes_{L_v}\mathbf{Z}\right)$.
Thus (\ref{eq4.3}) can be written in the form:
\begin{equation}\label{eq4.5}
\Sigma(\mathscr{A}_{V})\to  \prod_v K_0\left(\mathscr{A}_{V}\rtimes_{L_v}\mathbf{Z}\right).
\end{equation}

\smallskip
Both sides of (\ref{eq4.5}) are functions of a single positive matrix $B\in GL(2n, \mathbf{Z})$.   
However,  the LHS of  (\ref{eq4.5})  depends on the similarity class of 
$B$, while the RHS of (\ref{eq4.5})  depends   on the characteristic polynomial  of $B$. 
This observation is critical, since it puts the  elements of  $\Sha (A)$
into a one-to-one correspondence with the similarity classes of matrices having the same
characteristic polynomial.  The latter is an old problem of linear algebra and 
it is  known,  that the number of  such classes  is finite.  
They  correspond to  the ideal classes $Cl~(\Lambda)$ of an order   $\Lambda$ in  the field 
$K\cong\mathbf{Q}(\lambda_B)$, where  $\lambda_B$  is  the Perron-Frobenius eigenvalue of matrix $B$.  
Let  $Cl~(\Lambda)\cong \left(\mathbf{Z}/2^k\mathbf{Z} \right) \oplus Cl_{~\mathbf{odd}}(\Lambda)$ for $k\ge 0$. 
Using the Atiyah pairing between the K-theory and   the  K-homology, one gets   the following result.  
\begin{theorem}\label{thm4.2}
{\bf (\cite{Nik2})}
The Shafarevich-Tate group of an abelian variety $A(k)$
is a finite group given by the formulas:
\begin{equation}
\Sha (A)\cong
\begin{cases}
Cl~(\Lambda)\oplus Cl~(\Lambda), & \mbox{if $k$ is even,}  \cr
\left(\mathbf{Z}/2^k\mathbf{Z}\right) \oplus Cl_{~\mathbf{odd}}~(\Lambda)\oplus  Cl_{~\mathbf{odd}}~(\Lambda),  & \mbox{if $k$ is odd.} 
\end{cases}
\end{equation}
\end{theorem}

\subsection{Height functions \cite{Nik4}}
The height function is a measure of complexity of solutions 
of the diophantine equations.  For example, if $P^n(\mathbf{Q})$
is the $n$-dimensional projective space over $\mathbf{Q}$,  then the height of a
point $x\in P^n(\mathbf{Q})$ is given by the formula $H(x)=\max\{|x_0|,\dots,|x_n|\}$,
where $x_0,\dots,x_n\in\mathbf{Z}$ and $\mathbf{gcd}(x_0,\dots,x_n)=1$.
In general, if $V(\mathbf{Q})$ is a projective variety,  then  the height of
$x\in V(\mathbf{Q})$ is defined as  $H(x)=H(\phi(x))$,
where $\phi: V(\mathbf{Q})\hookrightarrow P^n(\mathbf{Q})$
is an embedding of $V(\mathbf{Q})$  into the  space 
$P^n(\mathbf{Q})$. 
The fundamental property of $H(x)$ says that the number of points $x\in V(\mathbf{Q})$, 
such that $H(x)\le T$,  is  finite for any constant $T>0$.  Thus one can 
define  a counting function $\mathcal{N}(V(\mathbf{Q}), T)=\# \{x\in V(\mathbf{Q}) ~|~ H(x)\le T\}$. 
Such functions depend on the topology of the underlying variety. 
In particular, if the counting function is bounded by a constant, then the set $V(\mathbf{Q})$
must be finite. 

One can use Theorem \ref{thm1.1} to define a height function on the variety $V(k)$. 
Indeed, let   $(K_0(\mathscr{A}_V), K_0^+(\mathscr{A}_V), \Sigma(\mathscr{A}_V))$ be the scaled dimension group   
[Blackadar 1986] \cite[Section 6.1]{B} of $\mathscr{A}_V$ given by Theorem \ref{thm1.1}. 
  Minkowski question-mark function maps the scale $\Sigma(\mathscr{A}_V)\subset K_0(\mathscr{A}_V)$
 into a  subgroup of the additive torsion group $\mathbf{Q}^n/\mathbf{Z}^n$.
  On the other hand, there exists a proper inclusion of the sets $V(k)\subseteq  K_0(\mathcal{A}_V)$
 unless the Shafarevich-Tate group $\Sha(V(k))$ is trivial  \cite[Lemma 3.1]{Nik4}. 
In other words,  $V(k)\subseteq \mathbf{Q}^n/\mathbf{Z}^n$. Clearly, this inclusion can be used to   define the height  $\mathscr{H}(x)$
of  a point $x\in V(k)$ as follows. 
\begin{definition}\label{dfn4.1}
$\mathscr{H}(1,\theta_1,\dots,\theta_n):=H(1,\frac{p_1}{q_1},\dots, \frac{p_n}{q_n})$,
where $H$ is the standard height of points of the  space  $P^n(\mathbf{Q})$.  
\end{definition}
Since  $V(k)\subseteq K_0(\mathscr{A}_V)$, 
 each point  $x\in V(k)$ get assigned a height $\mathscr{H}(x)$. 
The corresponding counting function is given by the formula $N(V(k), T)=\# \{x\in V(k) ~|~ \mathscr{H}(x)\le T\}$.
Denote by  $rk ~K_0(\mathscr{A}_V)$ the rank of the abelian group $K_0(\mathscr{A}_V)$. 
Our main results can be formulated as follows. 
\begin{theorem}\label{thm4.3}
{\bf (\cite{Nik4})}
Let $V(k)$ be an $n$-dimensional projective variety over the number field $k$. Then
\begin{equation}\label{eq4.7}
\log_2 N(V(k), T)\sim\begin{cases}
T^n, & \hbox{if} \quad rk ~K_0(\mathscr{A}_V) < n+1\cr 
n ~\log_2 ~T, & \hbox{if} \quad  rk ~K_0(\mathscr{A}_V) = n+1\cr
Const, & \hbox{if} \quad rk ~K_0(\mathscr{A}_V) > n+1.
\end{cases}
\end{equation}
\end{theorem}
\begin{remark}
Given that  $rk ~K_0(\mathscr{A}_V)=2\dim_{\mathbf{C}}\mathscr{M}_V$, 
one can derive from (\ref{eq4.7}) a rule of thumb:  ``moduli $\mathscr{M}_V$ kill the rational points of $V(k)$'',  i.e.
the growing dimension of the space $\mathscr{M}_V$ decrease  the number $|V(k)|$.  
\end{remark}

Let $\beta_i$ be the $i$-th Betti number of the $n$-dimensional variety $V(k)$. 
Theorem \ref{thm4.3}  and the Chern character formula imply the following finiteness result. 
\begin{corollary}\label{cor4.2}
{\bf (\cite{Nik4})}
The set $V(k)$ is finite,  whenever $\sum_{i=1}^{n} \beta_{2i-1}> n+1$. 
\end{corollary}

  
   
  

\bibliographystyle{amsplain}


\end{document}